\documentclass[12pt]{amsart}
\usepackage{dsfont}
\usepackage{listings}
\usepackage[utf8]{inputenc}
\usepackage{amsmath}
\usepackage{amssymb}
\usepackage{fullpage}
\theoremstyle{plain}
\newtheorem{theorem}{Theorem}[section]
\newtheorem{lemma}[theorem]{Lemma}

\newtheorem{corollary}[theorem]{Corollary}

\newtheorem{proposition}[theorem]{Proposition}

\theoremstyle{remark}
\newtheorem{remark}[theorem]{Remark}

\newtheorem{eg}[theorem]{Example}

\theoremstyle{definition}
\newtheorem{definition}[theorem]{Definition}

\newcommand{\ZZ}{\mathds{Z}}
\newcommand{\CC}{\mathds{C}}
\newcommand{\QQ}{\mathds{Q}}

\title{Finite groups of matrices over quadratic number fields.}

\author{
Daniil Yurshevich
}
\address{
Faculty of Mathematics and Informatics, Jagiellonian University, Kraków, Poland}
\email{daniil.yurshevich@student.uj.edu.pl}
\begin{document}

\maketitle

\small





\begin{abstract}
    In this paper we give an algorithm to determine all finite matrix groups  over a number field. Our algorithm is based on the representation theory of finite groups.
\end{abstract}

\section*{Introduction}
The aim of this paper is to give an algorithm to compute finite matrix groups over a number field.
Study of finite matrix groups with rational coefficients begun already in XIX century. Minkowski (\cite{Minkowski}) proved that the order of a finite subgroup of $\operatorname{GL}_n(\mathbb Q)$ is bounded by the constant
	\[M(n):=\prod_{p\in\mathbb P} p^{\lfloor\frac n{p-1}\rfloor+\lfloor\frac n{p(p-1)}\rfloor+\lfloor\frac n{p^2(p-1)}\rfloor+\ldots}.\]
Schur proved that the same bound holds for the order of any finite subgroup of $\operatorname{GL}_n(\mathbb C)$ under a weaker assumption that only traces of elements are rational.

Finite subgroups of $\operatorname{GL}_2(\mathbb Q)$ are well known, subgroups of $\operatorname{GL}_3(\mathbb Q)$ where classified in \cite{Tahara} (up to \cite{AG}).
In higher dimensions primitive irreducible maximal finite subgroups and maximal irreducible finite subgroups  of $\operatorname{GL}_n(\mathbb Q)$ were studied in numerous papers (see fi. \cite{BNW1, BNW2, BNW3, Dade, KP, Nebe, pp1,pp2,pp3,pp4, pp5, pp6} and the references cited there). Current version of Magma (\cite{Magma}) contains database Rational Maximal Finite Matrix Groups of dimensions $\le 31$ (it contains in total 354 groups).

Schur (\cite{Schur}) proved a similar bound $S(n,K)$ of the order of a finite group of matrices with traces in a number field $K$. If $K$ is a number field with class number equal 1, for instance this holds for $K=\mathbb Q$, any finite subgroup of $\operatorname{GL}_n(K)$ is conjugated to a subgroup of $\operatorname{GL}_n(\mathcal{O}_K)$.

Our research was motivated by applications in algebraic geometry, many examples of Calabi-Yau manifolds were constructed as a crepant resolution of singularities of a quotient $E^n/G$ of a Cartesian power $E^n$  of an elliptic curve $E$ by a finite group $G$ acting on $E^n$ and preserving the canonical form $\omega_{E^n}$. Any finite group $G\subset \operatorname{GL}_n(\mathcal O)$ of matrices with entries in the endomorphism ring  $\mathcal O=\operatorname{End}(E)$ of $E$ acts in the obvious way on $E^n$, the action preserve the canonical form when $G$ is contained in the special linear group
$\operatorname{SL}_n(\mathcal O)$. The endomorphism ring of an elliptic curve equals $\ZZ$ (ordinary) or an order in an imaginary quadratic field (Complex Multiplication).

All examples constructed in that way for finite subgroups $G\subset \operatorname{SL}_n(\mathbb Z)$  with $n\le 3$ are well studied \cite{AW,Burek,CS}. In an arbitrary dimension $n$ we have examples with an action on $E^n$ of $\ZZ_2^{n-1}$ (for an arbitrary $E$), $\ZZ/4^{n-1}$ (for elliptic curves with $j(E)=1728$), $\ZZ/3^{n-1}$ and $\ZZ/6^{n-1}$ (for elliptic curves with $j(E)=0$) (\cite{Burek2, CH}).

Similar construction exists also for the elliptic curve $E_7$ with the Weierstrass equation
\[y^2=4x^3+21x^2+28x.\]
The curve $E_7$ has a complex multiplication by
\[\tau=\zeta_7+\zeta_7^2+\zeta_7^4 = \frac12(-1+\sqrt{-7}).\]
The matrix
\[
\begin{pmatrix}
	0&0&1\\1&0&\tau+1\\0&1&\tau
\end{pmatrix}\]
is similar to the diagonal matrix $\operatorname{diag}(\zeta_7,\zeta_7^2,\zeta_7^4)$ and acts on $(E_7)^3$.
The quotient of $E_{7}^{3}$ by the cyclic group generated by this matrix has a crepant resolution of singularities which is a Calabi-Yau threefold (for details see \cite{Roan}). Elliptic curves with complex multiplication corresponds to orders in quadratic field, by the Stark-Heegner theorem there are nine quadratic number fields with class number one and thirteen elliptic curves with CM defined over $\mathbb Q$.

The paper is organized as follows,  in section \ref{s:prelim} we collect information on representations of finite groups that we shall use in this paper. In section \ref{sec:weil}. we define the Weil restriction of scalars, in section 4. the Schur index of a character and characterize characters afforded by a representation defined over a given number field. In Section 5. we review details of the Schur bound and the reduction method. Finally, section \ref{sec:alg}. contains description of our algorithm, summary of results for quadratic fields with class number equal one and $n=3$, details of computations for $\operatorname{GL}_{3}(\QQ[\sqrt{-19}])$ and .
$\operatorname{SL}_{3}(\QQ[\sqrt{-19}])$.
\section{Preliminaries}
\label{s:prelim}
    In this section we will introduce basic notions of representation theory of finite groups. Let $G$ be a finite group.

    \begin{definition} Linear representation of the group $G$ in a vector space $V$ is a homomorphism $\rho: G \rightarrow Aut(V)$ of the group $G$ into the space of automorphisms of $V$. Equivalently representation $\rho$ is given by a family $\{\rho_g \}_{g \in G}$ s.t. $\rho_{gh} = \rho_g \rho_h$ for all $g,h \in G$
    \end{definition}

	We call a representation \textit{faithful} if it is injective as homomorphism of groups. Space $V$ is called \textit{representation space}. In this paper we will only be interested in the case where $G$ is finite and vector space is finite dimensional so let's assume further that these conditions hold.

    \begin{definition}
        Let $\rho$ and $\rho'$ be linear representations of a Group $G$ in  vector spaces $V$ and $W$, respectively. A morphism from $\rho$ to $\rho'$  is a linear map  $\tau:V\longrightarrow W$ such that for each $g\in G$ we have  $\tau\rho_g = \rho'_g \tau$. If additionally $\tau$ is an isomorphism then we say that $\rho$ and $\rho'$ are isomorphic.
    \end{definition}

    Here are some examples of representations:

    \begin{enumerate}
        \item \textit{Trivial representation}. Take $V = \mathds{C}$ and $\rho_g = 1$ for all $g \in G$.
        \item \textit{Regular representation}. Take $V = \mathds{C}^{|G|}$ and define $\rho_g$ on basis $\{ e_h\}_{h\in G}$ of $V$ as $\rho_g(e_h) = e_{gh}$ for all $g, h \in G$.
        \item \textit{Direct sum of representations}. Having two representations $\rho$ and $\rho'$ in $V$ and $W$ we define the direct sum $\rho\oplus\rho'$ follows: representation space of $\rho \oplus \rho'$ is $V \oplus W$ and $(\rho \oplus \rho') (v,w) = \rho(v) \oplus \rho'(w)$.
    \end{enumerate}

    \begin{definition}
        Let $\rho$ be a linear representation of a group $G$ in a vector space $V$. The \textit{character} of $\rho$ is the function $\chi:G\longrightarrow\CC$ defined by $\chi(g) = trace(\rho_g)$.
    \end{definition}

    Trace of endomorphisms (matrices) is not multiplicative but it satisfies the following property: $trace(AB) = trace(BA)$ for any $A,B \in Aut(V)$. Consequently for any $g,h$ we have $trace(ghg^{-1}) = trace((gh)g^{-1}) = trace(g^{-1}(gh)) = trace(h) = trace(h)$ and so $\chi_G(g)$ depends only on conjugacy class of $g$.

    \begin{definition}
        Let $G$ be a group and $f: G \rightarrow \mathds{C}$ some function. We call $f$ class function if $f(g)$ depends only on conjugacy class of $g$ for  $g \in G$.
    \end{definition}

	In particular characters of linear representations are class functions and
    characters are invariants of isomorphism classes of representations. It means that isomorphic representations have the same characters. Another property of characters is that character of direct sum of representations is the sum of characters of summands.

    Let's now introduce the scalar product of class functions (in particular of characters): if $f$ and $h$ are class functions then $\langle f, h\rangle = \sum\limits_{g \in G}f(g) \overline{h(g)}$.

    \begin{lemma}\label{lem1}
        If $\rho$ and $\rho'$ are two representations of $G$ in $V$ and $W$ resp. then $\langle \chi_\rho, \chi_{\rho'}\rangle = \dim(\operatorname{Hom}_G(V, W))$, where  $\operatorname{Hom}_G(V, W)$ are $G$-linear maps between $V$ and $W$.
    \end{lemma}

\begin{definition}
	    \emph{A subrepresentation} of a given representation $\rho$ in a  vector space $V$ is representation $\rho'$ in some subspace $W$ of $V$ such that $\rho'_g = \rho_g |_W$ for all elements of $G$. We call representation \emph{irreducible} if it has no proper nontrivial subrepresentation.
\end{definition}

        \begin{theorem}[\mbox{Schur's lemma, \cite[Prop. 2.4]{Serre}}]

        Let $V$ and $W$ be vector spaces; and let $\rho_V$ and $\rho_W$ be irreducible representations of $G$ in $V$ and $W$.
        \begin{enumerate}
 \item If $V$ and $W$ are not isomorphic, then there are no nontrivial $G$-linear maps between them.
 \item If $V=W$ and $\rho_V = \rho_W$, then the only $G$-linear maps are the scalar multiples of the identity.
 \end{enumerate}

    \end{theorem}
    \begin{theorem}[\mbox{Maschke Theorem, \cite[Thm. 10.8]{CR}}]
        If $\rho$ is representation of $G$ and $\rho '$ is subrepresentation of $\rho$, then there is $\rho ''$ such that $\rho \cong \rho ' \oplus \rho ''$.
    \end{theorem}

    \begin{corollary}\label{cor1}
        If $\rho$ is representation of $G$ then there exist $\rho_1, \rho_2, \dots, \rho_n$ irreducible representations of $G$ such that $\rho \cong  \rho_1 \oplus \dots \oplus \rho_n$
    \end{corollary}

    \begin{corollary}\label{cor2}
        If $\rho$ is irreducible representation and $\rho'$ is some representation of $G$ then number of $\rho_i$ isomorphic to $\rho$ in decomposition if $\rho'$ is equal to $\langle \chi_\rho, \chi_{\rho'}\rangle$.
    \end{corollary}

    This decomposition is unique up to permutation of $\rho_i$ and isomorphism of the summands. More precisely if $\oplus_{i=1}^n\rho_i \cong \oplus_{i=1}^m\rho'_i$ for irreducible representations $\rho_{i}, \ i \in \{1, \dots , n\}$ and $\rho'_{i}, \ i \in \{1, \dots , m\}$ then $n=m$ and there exist permutation $\sigma$ of $\{1, \dots, n \}$ such that $\rho_i = \rho_{\sigma(i)}$.

    Now, by property of character of direct sum and  Corollary \ref{cor1} character of every representation is equal to the sum of characters of irreducible representations with non-negative integer coefficients. Moreover coefficients can be calculated using scalar product of characters. So it is enough to study characters of irreducible representation's.

    \begin{corollary}\label{cor3}
        The number of irreducible representations of $G$ is finite for every finite group $G$.
    \end{corollary}


    \begin{theorem}\label{the1}
       Let $\chi_1, \ldots \chi_n$ be characters of all irreducible representations of $G$ (by Corollary \ref{cor3} the number of irreducible representations if finite). Then $\langle \chi_i, \chi_j \rangle = 0$ if $i \neq j$ and $\langle \chi_i, \chi_j \rangle = 1$ otherwise, i.e. characters of irreducible representations form orthonormal system in the space of class functions on $G$. Moreover, they form orthonormal basis of class functions on $G$, i.e. every class function can be written as linear combination of characters of irreducible representations with complex coefficients.
    \end{theorem}

    \begin{corollary}\label{cor4}
        The number of irreducible representations of every finite group $G$ is equal to class number of $G$.
    \end{corollary}


    From corollary \ref{cor4} it follows that we can form a $c \times c$ matrix, where $c$ is class number of $G$, whose rows represent characters of irreducible representations. This matrix is called \textit{character table}.

	Decomposition of a representation of a finite group into a direct sum of irreducible representations is not unique, however unique is the ``canonical decomposition'' obtained by collecting together isomorphic irreducible representations.

       Let $\rho$ be a representation of a finite group $G$. Let $\chi_{1},\dots,\chi_{c}$ be distinct characters of irreducible representations $\rho_{1},\dots,\rho_{c}$ of $G$.
       Consider any decomposition $\rho=\tau_{1}\oplus\cdots\oplus\tau_{m}$ of $\rho$ into a direct sum of irreducible subrepresentations. Denoting by
       $\rho_{i}=\bigoplus_{\chi_{\tau_{i}}=\chi_{i}}\tau_{i}$ the direct sum of representations $\tau_{j}$'s with the character equal $\chi_{i}$ we get decomposition
       \begin{equation}
       	\rho=\rho_{1}\oplus\cdots\oplus\rho_{s}.\label{eq:canon}
       \end{equation}
\begin{definition}
	We shall call the decomposition \eqref{eq:canon} the \emph{canonical decomposition} of the representation $\rho$.
\end{definition}

     \begin{theorem}[\text{\cite[Thm. 8]{Serre}}]\label{the1}
     Let \[\rho=\rho_{1}\oplus\cdots\oplus\rho_{s}\] be a canonical decomposition of a representation $\rho$ of a finite group $G$.
       For  $i=1,\dots,s$ the representation
       $\rho_{i}$ is uniquely determined, in particular it doesn't depend on the choice of the initially chosen decomposition of $\rho$ into irreducible subrepresentations.

       More precisely the projection of $\rho$ onto $\rho_{i}$ is given by:

\[
    \frac{n_{i}}{|G|}\sum\limits_{g \in G} \overline{\chi_{i}(g)}\rho(g).
\]
where $n_{i}$ is the degree of representation $\rho_{i}$.

In particular if representation $\rho$ is defined over $K$ then $V_\rho$ will be defined over $K$ as well.
\end{theorem}
\begin{definition}
A representation $\rho:G\longrightarrow\operatorname{GL}(V)$ of  a finite group $G$ is \emph{realizable over a field} $K\subset\CC$ iff  there exists a basis of $V$ such that matrices of images of all elements of $G$ have coefficients in $K$.

Equivalently, a representation $\rho:G\longrightarrow\operatorname{GL}(V)$  is realizable over field $K\subset\CC$ iff there exists a homomorphism $\tau:G\longrightarrow\operatorname{GL}(W)$, where $W$ is a vector space over $K$ with an isomorphism $\psi:W\otimes_{K}\CC \longrightarrow V$ such that
\[\rho_{g} = \psi\circ\tau_{g}\otimes_{K}\CC\circ\psi^{-1}.\]

\end{definition}
\begin{corollary}
	If a representation $\rho$ of a finite group is realizable over a field $K\subset\CC$ then all components of the canonical decomposition of $\rho$ are also defined over $K$.
\end{corollary}
The above corollary doesn't hold for irreducible components of a representation, in fact every irreducible representation of a finite group is a component of its regular representation which is defined over $\QQ$.

\begin{eg}\label{ex17}
	The matrix $M:=\begin{pmatrix}
		0&1\\-1&-1
	\end{pmatrix}$ generates a subgroup of $\operatorname{GL}_{3}(\CC)$ of order 3, so it defines a degree 2 representation of the group  $G:=C_{3}$ which is realizable over $\QQ$. This representation is a direct sum of two degree one representations determined by the degree 3 primitive roots of unity $\zeta_{3}$ and $\zeta_{3}^{3}$, which are clearly not realizable over $\QQ$.
\end{eg}
\begin{eg}\label{ex18}
	The group generated by matrices
	$\begin{pmatrix}
		0&-1\\1&0
	\end{pmatrix}, \begin{pmatrix}
	i&0\\0&-i
	\end{pmatrix}$ is the quaternion $Q_{8}$ hence it defines a faithful representation $\rho$ of $Q_{8}$ of degree 2 which is not realizable over $\QQ$, in fact this representation is realizable over a field $K\subset \CC$ iff -1 is a sum of two squares in $K$ (see \cite[exerc. 12.3]{Serre}).
	On the other hand the representation $\rho^{\oplus2}$ is realizable over $\QQ$,
	more precisely the matrices
	$\begin{pmatrix}
		0&0&-1&0\\0&0&0&-1\\1&0&0&0\\0&1&0&0
	\end{pmatrix}, \begin{pmatrix}
	0&-1&0&0\\1&0&0&0\\0&0&0&1\\0&0&-1&0
	\end{pmatrix}$ generate group isomorphic to $Q_{8}$ and it determines a representation of $Q_{8}$ isomorphic to $\rho^{\oplus2}$ which is defined over $\QQ$.
\end{eg}
\section{Weil Restriction of scalars}
\label{sec:weil}

\begin{definition}
    A splitting field of a finite group $G$ is a field $K\subset\CC$ such that every representation of $G$ over $\mathds{C}$ can be realized over $K$, equivalently every irreducible representation can be realized over $K$.
\end{definition}

Clearly, a splitting field of a group $G$ contains the field generated by character value. A splitting field of a finite group  is not unique, in fact there can be several distinct minimal splitting fields (cf. Ex. 18).
The Brauer Theorem (\cite[Thm. 24]{Serre}) gives a sufficient condition for a splitting field
\begin{theorem}[Brauer Theorem, \mbox{\cite[Thm. 24]{Serre}}]
	Let $m$ be the exponent of a finite group $G$. If a field $K\subset\CC$ contains a primitive root of unity of degree $m$ then $K$ is a splitting field of $G$.
\end{theorem}

Let  $K \subset L$ be a field extension of degree $m:=[L:K]$.
Any representation
$\rho:G\longrightarrow \operatorname{GL}(V)$
of a finite group $G$ realizable over field $L$ is determined by a homomorphism
$\tau:G\longrightarrow \operatorname{GL}(W)$
for a vector space $W$ over $L$ with an isomorphism $\psi:W\otimes_{L}\CC\longrightarrow V$. The vector space $W$ has a natural structure of a vector space over $K$, which we shall denote by $W_{K}$. Then the group of automorphisms of $W_{K}$, which is equal to the group of $K$-linear automorphisms of $W$, contains the group of automorphisms of $W$
\[\operatorname{GL}(W)\subset \operatorname{GL}_{K}(W) = \operatorname{GL}(W_{K}).\]
Consequently, the homomorphism $\tau$ induces a homomorphism
$\tau:G\longrightarrow \operatorname{GL}(W)$ and hence a representation
\[\operatorname{Res}_{L/K}\rho: G\longrightarrow V_{K}, \ \ \text{ where } V_{K} = W\otimes _{K}\CC.\]

\begin{definition}
	The representation $\operatorname{Res}_{L/K}\rho$ is called \emph{the Weil restriction of scalars} of the representation $\rho$ from $L$ to $K$.
\end{definition}
Since $V_{K}= W\otimes_{K}\CC\cong W\otimes_{K}(L\otimes _{L}\CC)\cong (W\otimes_{K}L)\otimes _{L}\CC\cong (L\otimes_{K}W)\otimes _{L}\CC\cong L\otimes_{K}(W\otimes _{L}\CC)\cong L\otimes _{K}V$ we get
\[\dim_{\CC}V_{K} = [L:K]\dim_{\CC}V,\]
which means that the degree of the Weil restriction of a representation $\rho$ equals the degree of the original representation $\rho$ multiplied by the degree of the fields extension.

Weil restriction can be also described more explicitly at the level of matrix representation. Let $M\in \operatorname{GL}_{n}(L)$ be an $n\times n$  invertible matrix with entries in $L$, it defines an automorphism $\phi_{M}:L^{n}\longrightarrow L^{n}$. The field $L$ is a $K$-vector space of dimension $m$, choice of a basis of $L$ over $K$ yields an isomorphism between $L$ and   $K^{m}$. Since $\phi_{M}$ is also an automorphism of vector spaces over $K$ it determines an automorphism $\phi_{M}^{K}:K^{mn}\longrightarrow K^{mn}$ and hence also an invertible $mn\times mn$ matrix $\operatorname{Res}_{L/K}\in\operatorname{GL}_{mn}(K)$. Definition of $\operatorname{Res}_{L/K}$ depends on a choice of a basis of $L^{n}$ as a vector space over $K$ - which is equivalent to a choice of an isomorphism of $L^{n}$ and $K^{mn}$.

\begin{definition}
	The matrix $\operatorname{Res}_{L/K}(M)$ is called the Weil restriction of scalars of $M$ from $L$ to $K$.
\end{definition}

Assume then $\{b_{1},\dots,b_{m}\}$ is an ordered base of $L/K$ and $v_{1},\dots,v_{n}$ is the natural basis of $V$ over $L$. Then the set of products $(b_{i}v_{j})_{i=1,\dots,m,j=1,\dots,n}$ is a basis of $V/K$.
Assume that the matrix of an endomorphism $F\in\operatorname{End}(V)$ in the chosen base equals $A=(a_{ij})_{i,j=1,\dots,n}$.
Now, $F(b_{i}v_{j}) = b_{i}F(v_{j}) = \sum_{p=1}^{n}a_{p,j}b_{i}v_{p}$.

For each element $a \in L$, multiplication by $a$
\[L\ni x\longmapsto ax\in L\]
is a $K$-endomorphism of $L$, denote by $\operatorname{Res}_{L/K}(a)$ the matrix of this endomorphism in the base $b_{i}$. Then
\[a_{p,j}b_{i} = \sum_{q=1}^{m} \operatorname{Res}_{L/K}(a_{p,j})_{q,i}b_{q}\]
and finally
\[F(b_{i}v_{j}) = \sum_{p=1}^{n}\sum_{q=1}^{m} \operatorname{Res}_{L/K}(a_{p,j})_{q,i}b_{q}v_{p}\]

The explicit matrix of $F$ depends on the order of the base of $V$ over $K$, for the ordering $\{b_{1}v_{1},\dots,b_{m}v_{1},\dots,b_{1}v_{n},b_{m}v_{n}\}$, we get the block matrix
\[\left(\begin{array}{cccc}
	\operatorname{Res}_{L/K}(a_{1,1})&\operatorname{Res}_{L/K}(a_{1,2})&\cdots& \operatorname{Res}_{L/K}(a_{1,n})\\
	\operatorname{Res}_{L/K}(a_{2,1})&\operatorname{Res}_{L/K}(a_{2,2})&\cdots& \operatorname{Res}_{L/K}(a_{2,n})\\
	\vdots&\vdots&\ddots&\vdots\\
	\operatorname{Res}_{L/K}(a_{n,1})&\operatorname{Res}_{L/K}(a_{n,2})&\cdots& \operatorname{Res}_{L/K}(a_{n,n})
\end{array}\right)\]
\begin{proposition}
	For a matrix $A(a_{ij})_{i,j=1,\dots,n}\in\operatorname{GL}_{n}(L)$ we have (in a appropriately ordered base)
	\[\operatorname{Res}_{L/K}(A)=(\operatorname{Res}_{L/K}(a_{i,j}))_{i,j=1,\dots,n}.\]
\end{proposition}
Let $d=[K[a]:K]=\deg_{K}a$ be the degree of $a$ over $K$, denote $r:=\frac md=[L:K[a]]$ and choose an ordered base $\{f_{1},\dots,f_{r}\}$  of $L$ over $K[a]$. Then
\[a(a^{k}f_{l})=a^{k+1}f_{l}, \text{ for }k=0,\dots,d-2,\]
and
\[a(a^{d-1}f_{l})= -(c_{d-1}a^{d-1}f_{l}+\dots+c_{0}f_{l}),\]
where $P(X) = X^{d}+c_{d-1}X^{d-1}+\dots+c_{0}$ is the minimal polynomial of $a$ over $K$. Denoting by $A_{P}$ the companion matrix of $P$ we get that in the base $\{f_{1},a f_{1},\dots,a^{d-1}f_{1},\dots,f_{r},\dots,a^{d-1}f_{r}\}$
\[\operatorname{Res}_{L/K}(a)=\operatorname{diag}(A_{P},\dots,A_{P})\]
If the extension $L/K$ is Galois with the Galois group $\operatorname{Gal}(L/K)=\{\sigma_{1},\dots,\sigma_{m}\}$ the last matrix is similar to
\[\operatorname{diag}(\sigma_{1}(a),\dots,\sigma_{m}(a)).\]

For a matrix $A$ with entries in the field $L$ and any automorphism $\sigma\in\operatorname{Aul}(L)$ we denote by $A^{\sigma}$ the matrix obtained from $A$ by application of $\sigma$ to each entry.

\begin{proposition}
	Let $A\in GL_{n}(L)$ be an invertible matrix over a field $L$, assume that $L$ is a finite Galois extension of $K$ with the Galois group $\{\sigma_{1},\dots,\sigma_{m}\}$. Then the matrices
	\[\operatorname{Res}_{L/K}(A)\ \ \text{ and }\ \ \operatorname{diag}(A^{\sigma_{1}},\dots,A^{\sigma_{m}})\]
	are similar.
\end{proposition}
\begin{corollary}
	\[\operatorname{trace}(\operatorname{Res}_{L/K}(A)) = \operatorname{trace}_{L/K}(\operatorname{trace}(A)) \]
\end{corollary}
Clearly, for any automorphism $\sigma\in\operatorname{Gal}_K(L)$  the map $A\longmapsto A^{\sigma}$ defines an automorphism of $\operatorname{GL}_{n}(L)$, hence for any representation $\rho$ realizable over $L$ we can define a representation $\rho^{\sigma}$.
\begin{corollary}
	For any representation $\rho$ realizable over a field $L$ which is a finite Galois extension of a field $K$ we have
	\[\chi_{\operatorname{Res}_{L/K}\rho} = \operatorname{trace}_{L/K}(\chi_{\rho})\]
\end{corollary}
\begin{theorem}\label{the6}
    Every representation $\mu$ realizable over $K$ containing irreducible representation $\rho$ has to contain representation isomorphic to  $Res_{L/K}(\rho)$ as well, where $L$ is some, minimal in sense of inclusion, field over which representation $\rho$ is realizable.
\end{theorem}

\section{Schur index}
\label{sec:schurind}

In the previous section we described two necessary conditions which a representation $\rho$ must satisfy to be defined over a number field $K$
\begin{itemize}
	\item values of the character of $\rho$ must belong to $K$,
	\item conjugated by an automorphism of $K$ irreducible representations must belong to $\rho$ with the same multiplicity.
\end{itemize}
As we noticed in Ex. \ref{ex18} these conditions are not sufficient, there is a degree 2 representation of quaternion group $Q_{8}$ which is defined over $\QQ[i]$ but is not realizable over $\QQ$. As this is the only representation of $Q_{8}$ of degree 2, it is isomorphic to its complex conjugate. Consequently a direct sum of two copies of $\rho$ is a representation of degree 4 realizable over $\QQ$.

Let $K\subset \CC$ be a subfield of the field of complex numbers, $\chi$ a character of a finite group $G$. Denote by $K(\chi)$ extension of the field $K$ by values of the character $\chi$.

\begin{definition}\cite[Def.~41.4]{CR}
	The \emph{Schur index} $m_{K}(\chi)$ of the character $\chi$ is the minimum
	of extensions degrees $[L:K(\chi)]$ taken over all field $L$ such that
	$K\subset L\subset \CC$ and $\chi$ is a character of a representation  of $G$ realizable over $L$.
\end{definition}
\begin{remark}
	We have $m_{K}(\chi)=1$ iff the character $\chi$ affords a representation realizable over $K(\chi)$.
\end{remark}
Schur index has alternative description based on the Wedderburn Theorem (\cite[Thm. 26.4]{CR}, \cite[sect. 12.2]{Serre}).

Let $L$ be a field extension such that $[L:K(\chi)]=m_{K}(\chi)$ and let $\rho$ be a representation of $G$ realizable over the field $L$. Then the Weil restriction $\operatorname{Res}_{L/K(\chi)}(\rho)$ is a representation of $G$ realizable over $K(\chi_{i})$.

\begin{corollary}[Schur, \mbox{\cite[p. 479]{CR}}]
	Let $L$ be be a finite Galois extension of $\QQ$, which is a splitting field of a finite group $G$ and let $\chi_{1},\dots,\chi_{c}$ be all irreducible characters of $G$.
	Then a character $\chi=n_{1}\chi_{1}+\dots+n_{c}\chi_{c}$  affords a representation of $G$ realizable over a subfield $K\subset L$ iff the following two conditions hold
	\begin{enumerate}
		\item $m_{K}(\chi_{i})\mid n_{i}$,
		\item if $\chi_{i}=\chi_{j}^{\sigma}$ is a conjugate of $\chi_{j}$, $\sigma\in\operatorname{Gal}_K(L)$, then $n_{i}=n_{j}$.
	\end{enumerate}
\end{corollary}

\section{Schur bound}
\label{sec:schurbound}

In this section we shall recall the Schur bound of the order of a finite group of matrices over a given number field, we shall follow \cite{Guralnick, Serre2}.

For an integer $m>1$ denote by $\zeta_m = e^{\frac{2 \pi i}{m}}$ a (fixed) degree $m$ primitive root of unity.

\begin{definition}
For a prime number $l$ define
\[
m(K, l) = \min\{ m \ge 1 | K \cap \mathds{Q}(\zeta_{l^m}) =  K \cap \mathds{Q}(\zeta_{l^{m+1}}) = \dots\}
\]

\end{definition}
The chain of inclusions
\[\mathds{K} \cap \mathds{Q}(\zeta_{l}) \subset \mathds{K} \cap \mathds{Q}(\zeta_{l^2}) \subset \mathds{K} \cap \mathds{Q}(\zeta_{l^3}) \subset \dots\]
stabilizes since $[K:\mathds{Q}]$ is finite, so $m(K,l)$
is well-defined.

Since
\[[K\cdot\QQ(\zeta_{l^m}):\QQ] = \dfrac{[K:\QQ] [\QQ(\zeta_{l^m}):\QQ]} {[K\cap\QQ(\zeta_{l^m}):\QQ]}\]
and
\[[K\cdot\QQ(\zeta_{l^{m+1}}):\QQ] = \dfrac{[K:\QQ] [\QQ(\zeta_{l^{m+1}}):\QQ]} {[K\cap\QQ(\zeta_{l^{m+1}}):\QQ]}\]
using the tower law we get for $m\ge1$
\[[K\cap\QQ(\zeta_{l^{m+1}}):K\cap\QQ(\zeta_{l^m})] [K\cdot\QQ(\zeta_{l^{m+1}}):K\cdot\QQ(\zeta_{l^m})] = l.\]

\begin{lemma}
	If $l$ is an odd prime then the sequence $\mathds{K} \cap \mathds{Q}(\zeta_{l^m})$ stabilizes at the first equality.
\end{lemma}
\begin{proof}
Assume to the contrary that
\[K\cap \QQ(\zeta_{l^m}) = K\cap \QQ(\zeta_{l^{m+1}}) \subsetneq K\cap \QQ(\zeta_{l^{m+2}}).\]
Then $K\cap \QQ(\zeta_{l^{m+2}})$ is a subfield in $\QQ(\zeta_{l^{m+2}})$ which is not contained in $\QQ(\zeta_{l^{m+1}})$, hence the degree of the extension $[K\cap\QQ(\zeta_{l^{m+2}}):\QQ]$ is divisible by $l^{m+1}$. Now the degree of $[K\cap\QQ(\zeta_{l^{m+2}})\cap\QQ(\zeta_{l^{m+1}}):\QQ] = [K\cap\QQ(\zeta_{l^{m+1}}):\QQ]$ is divisible by $l^{m}$, consequently
$K\cap\QQ(\zeta_{l^{m+1}})$ is not a subfield of $\QQ(\zeta_{l^{m}})$ - contrary to
$K\cap \QQ(\zeta_{l^m}) = K\cap \QQ(\zeta_{l^{m+1}})$.

\end{proof}

Consequently, if $K\cap\QQ(\zeta_{l^{m+1}})\supsetneq K\cap\QQ(\zeta_{l^m})$
we have $K\cdot\QQ(\zeta_{l^{m+1}})=K\cdot\QQ(\zeta_{l^m})= \dots = K\cdot\QQ(\zeta_{l})$ and
$\zeta_{l^{m+1}} \in K(\zeta_{l})$.
So, we get another description of the number $m(K,l)$ for an odd prime $l$
\begin{corollary}\label{cor:4.3}
\[m(K,l)=\max\{d\ge1: \zeta_{l^{d}}\in K(\zeta_{l})\}.\]
\end{corollary}
The last formula gives an easy a priori bound $m\le v_{l}(\deg(K))$ of $m$ by the $l$-adic valuation of the degree of the field $K$.

The situation is different in the case of $l=2$, the reason is that for $m\ge3$ the cyclotomic field $\QQ(\zeta_{2^{m}})$ of degree $\phi(2^{m})=2^{m-1}$ contains three subfields of degree $2^{m-2}$ corresponding to the three elements of order 2 in the Galois group $\operatorname{Gal}(\QQ(\zeta_{2^{m}})/\QQ) \cong (\ZZ/2^{m}\ZZ)^\ast\cong \ZZ/2\ZZ\oplus \ZZ/{2^{m-2}}\ZZ$, explicitly they are
\[\QQ(\zeta_{2^{m-1}}), \QQ(\zeta_{2^{m}}+\bar \zeta_{2^{m}}), \QQ(\zeta_{2^{m}}-\bar \zeta_{2^{m}}).\]

If $m(K,2)=1$ then for any $m\ge1$ we have $K\cap\QQ(\zeta_{2^{m}})=\QQ$.
Similarly, if $m(K,2)=2$ then $i\in K$ and for any $m\ge2$ we have $K\cap\QQ(\zeta_{2^{m}})=K[\zeta_{4}]=K[i]$.

On the other hand if $m(K,2)\ge 3$ then $K\cap\QQ(\zeta_{2^{m(K,2)}})\not\subset\QQ(\zeta_{2^{(m(K,2)-1)}})$, hence
$K\cap\QQ(\zeta_{2^{m(K,2)}}) = \QQ(\zeta_{2^{m(K,2)}}))$ or
$K\cap\QQ(\zeta_{2^{m(K,2)}}) = \QQ(\zeta_{2^{m(K,2)}} + \bar \zeta_{2^{m(K,2)}})$ or $K\cap\QQ(\zeta_{2^{m(K,2)}}) = \QQ(\zeta_{2^{m(K,2)}} - \bar \zeta_{2^{m(K,2)}})$.
In the first case we get $K\cap\QQ(\zeta_{2^{m}}) = \QQ(\zeta_{2^{m}}) $ for $1\le m\le m(K,2)$.
In the second case we get $K\cap\QQ(\zeta_{2^{m}}) = \QQ(\zeta_{2^{m}}+\bar \zeta_{{2^{m}}})$ for $1\le m\le m(K,2)$. Finally in the third case we get
$K\cap\QQ(\zeta_{2^{m}}) = \QQ(\zeta_{2^{m}} - \bar \zeta_{2^{m}})$ for $m=m(K,2)$ and $K\cap\QQ(\zeta_{2^{m}}) = \QQ(\zeta_{2^{m}}+\bar \zeta_{2^{m}})$ for $1\le m< m(K,2)$.

We have an alternative description of $m(K,2)$
\begin{proposition}
	Let \[m_0=\max\{d\ge1: \zeta_{2^{d}}\in K(\zeta_{4})=K(i)\}.\]
	Then
	\[m(K,2)=\begin{cases}
		m_0,\quad&\text{when }m_0\ge3\\
		2, &\text{when }m_0=2 \text{ and } i\in K\\
		1, &\text{when }m_0=2 \text{ and } i\not\in K
	\end{cases}\]
\end{proposition}

\begin{eg}
	Let $K:=\QQ(\sqrt{2})$, then $\QQ(\zeta_{2})=\QQ$, $\QQ(\zeta_{4})=\QQ(i)$, $\QQ(\zeta_{8})=\QQ(i,\sqrt{2})$ and so
	\[K\cap\QQ(\zeta_{2})=\QQ=K\cap\QQ(\zeta_{4}) \subsetneq K=K\cap\QQ(\zeta_{8}) = \cdots\]
	We conclude that $m(K,2)=3$ but $\zeta_{8}\not\in K(\zeta_{2})$.
\end{eg}

Furthermore we define another invariant $t(K,l)$
\[
    t(K, l) = [\mathds{Q}(\zeta_{l^{m(K,l)}} ) : K \cap \mathds{Q}(\zeta_{l^{m(K,l)}} )] = [K(\zeta_{l^{m(K,l)}} ) : K ]\]
    For an odd prime $l$ we can repeat the above arguments to get
    \[t(K, l) = [K(\zeta_{l}):K].\]
Again, for $l=2$ the formula is slightly more complicated, if $m(K,2)\ge 3$ we get a sequence of strict inclusions
\[K\cap\QQ(\zeta_{4}) \subsetneq K\cap\QQ(\zeta_{8}) \subsetneq\dots \subsetneq K\cap\QQ(\zeta_{2^{m(K,2)}})\]
and hence a sequence of equalities
\[K(\zeta_{4}) = K(\zeta_{8}) =\dots = K(\zeta_{2^{m(K,2)}}) \]
Finally we get in the case  of $l=2$, $m(K,2)\ge 3$
    \[t(K, 2) = [K(\zeta_{4}):K]=[K(i):K].\]
If $m(K,2)=1$ then $K\cap\QQ(\zeta_{2})=\QQ=\QQ(\zeta_{2})$ and if $m(K,2)=2$ then $K\cap\QQ(\zeta_{4})=\QQ(i)=\QQ(\zeta_{4})$, in both cases we get $t(K,2)=1$.
\begin{definition}\label{def:schurbound}
The \emph{Schur number} of a number field $K$ is
\[
S(n,K) = 2^{n - \lfloor \frac{n}{t(K,2)} \rfloor} \prod\limits_l
l^{\left(m(K, l) \lfloor \frac{n}{t(K,l)}\rfloor+ \sum\limits_{i = 1}^{\infty} \lfloor \frac{n}{l^i t(K,l)}\rfloor\right)}
\]
\end{definition}
Denote by $S(n,K,l) = S(n,K)_l$ the $l$-part of $S(n,K)$, for $l$ an odd prime
\[S(n,K,l) =
l^{\left(m(K, l) \lfloor \frac{n}{t(K,l)}\rfloor+ \sum\limits_{i = 1}^{\infty} \lfloor \frac{n}{l^i t(K,l)}\rfloor\right)},\]
while for $l=2$ we have
\[S(n,K,2) =
2^{\left(n - \lfloor \frac{n}{t(K,2)} \rfloor + m(K, 2) \lfloor \frac{n}{t(K,2)}\rfloor+ \sum\limits_{i = 1}^{\infty} \lfloor \frac{n}{2^i t(K,2)}\rfloor\right)}.\]

Schur (\cite{Schur}) proved the following bound of the order of a finite matrix group over a number field.
\begin{theorem}[Schur's bound, \cite{Guralnick}]
	Let $G$ be a finite subgroup of $\operatorname{GL}_{n}(\CC)$ such that traces of all elements of $G$ belongs to a fixed number field $K$. Then
	the order of $G$ divides $S(n,K)$.
\end{theorem}

The corresponding bound of the order of a matrix $l$-group is sharp and an analogon of the Sylow theorem holds
\begin{proposition}
	Let $l$ be a prime and let $G$ be a finite $l$-subgroup of $Gl_n(K)$. Then $|G| \le S(n,K,l)$.

	Moreover this bound is sharp and if $H$ is an $l$-subgroup of maximal order then there exist $M \in Gl_n(K)$ such that $G$ is a subgroup of $mHm^{-1}$.
\end{proposition}
\begin{lemma}
	Assume that  $K$ is the field of fraction of a Principal Ideal Domain $\mathcal O\subset K$ and let $G\subset\operatorname{GL}_n(K)$ be a finite matrix group.
	Then $G$ is conjugate in $GL_n(K)$ to a subgroup of $GL_n(\mathcal{O}_K)$.
\end{lemma}
\begin{proof}
Consider the finitely generated $\mathcal O$-module of rank $n$
\[M := \sum_{g\in G} g\mathcal O^n.\]
Since $\mathcal O$ is Principal Ideal Domain  the
$\mathcal O$-module $M$ is isomorphic to $\mathcal O^n$. It means that there exists $A\in\operatorname{GL}_n(K)$ such that $M = A\mathcal{O}^n$. The module $M$ is stable under the action of $G$, i.e. for any $g\in G$ we have $gM=M$ and consequently $A^{-1}gA\in \operatorname{GL}_n(\mathcal O)$. The group $A^{-1}GA$ is a subgroup of $\operatorname{GL}_n(\mathcal O)$ conjugate to $G$.
\end{proof}

 For a number field $K$ we shall denote by $\mathcal O_K\subset K$ the ring of integers in $K$ i,e, the integral closure of $\ZZ$ in $K$.
\begin{proposition}\label{prop2}
    If $G$ is finite subgroup of $GL_n(K)$ and $\mathcal{O}_K$ is Principal Ideal Domain then $G$ is conjugate in $GL_n(K)$ to a subgroup of $GL_n(\mathcal{O}_K)$.
\end{proposition}
If $K$ is an arbitrary number field then $\mathcal O_K$ may not be a Principal Ideal Domain, it is however always a Dedekind domain. Consequently, for any prime ideal $\mathfrak p\subset\mathcal O_K$ the localization $\mathcal O_{K,\mathfrak p}$ is a Principal Ideal Domain and any finite subgroup of $\operatorname{GL}_n(K)$ is conjugate to a subgroup in $\operatorname{GL}_n(\mathcal O_{K,\mathfrak p})$.

The ideal $\mathfrak p$ has norm
 $\mathcal{N}(\mathfrak p) := \mathcal{O}_K/\mathfrak{p}$. Reduction modulo the maximal ideal $\mathfrak{p}\mathcal O_{K,\mathfrak p}$ of $\mathcal{O}_{K,\mathfrak p}$ yields a homomorphism:

\[
    Gl_n(\mathcal{O}_{K,\mathfrak p}) \rightarrow Gl_n(\mathds{F}_{\mathcal{N}(\mathfrak p)})
.\]
If $\mathfrak p\subset\mathcal O_K$ is a prime ideal, then $\mathfrak p\cap \ZZ=p\ZZ$, where $p$ is a rational prime. In this situation $p\mathcal O_{K,\mathfrak p} = \mathfrak p^e$, where $e$ is a positive integer called ramification index of $\mathfrak p$.

\begin{lemma}[\mbox{[Lemma 9 from \cite{Guralnick}]}] \label{lem3}
    The kernel of the reduction homomorphism above has at most $p$-torsion. Any torsion element $g$ in the kernel has order $p^i$ for some $p^i\le \frac{ep}{p-1}$.
\end{lemma}
%

\begin{corollary}
	Let $K$ be a number field, $\mathfrak p\subset\mathcal O_K$ a prime ideal with ramification index $e<p-1$. Then the kernel

	\[\ker\biggl(\operatorname{GL}_n(\mathcal O_{K,\mathfrak p}) \longrightarrow
	\operatorname{Gl}_n(\mathds{F}_{\mathcal{N}(\mathfrak p)}) \biggr)
	\]
	is torsion free. In particular every finite subgroup of
	$\operatorname{GL}_n(K)$ is isomorphic to a subgroup of
 $\operatorname{GL}_n(\mathds {F}_{\mathcal (N)(\mathfrak p)})$.
\end{corollary}

\begin{proposition}
	Let $\rho_1$ and $\rho_2$ be two faithful representations of a finite group $G$ in the same vector space $V$ with conjugate images in $\operatorname{GL}(V)$.
	Then there is an automorphism $\phi\in \operatorname{Aut}(G)$ such that representations $\rho_1$ and $\rho_2\circ\phi$ have equal character.
\end{proposition}

\begin{proof}
	Replacing $\rho_2$ with an isomorphic (conjugate) representation, we may assume that images of $\rho_1$ and $\rho_2$ are equal. Now, $\rho_1$ and $\rho_2$ are isomorphisms of $G$ onto the same subgroup of $\operatorname{GL}(V)$, and we can define an automorphism $\phi\in\operatorname{Aut}(G)$ by $\phi:=\rho_2 ^{-1}\circ \rho_1$.
\end{proof}

\section{Algorithm}
\label{sec:alg}


    In this section we shall describe main steps of the algorithm.
    Given a number field $K$ and a positive integers  compute all finite subgroups in $\operatorname{GL}_n(K)$. In each step of the algorithm we refer to an appropriate function in Magma code (sec. \ref{sec:magma})

    \begin{enumerate}\def\labelenumi{\arabic{enumi}.}
        \item  \texttt{schurBound:} \\
        Compute the Schur bound $S(n,K)$:  corollary \ref{cor:4.3} gives an easy explicit bound  $l_0$ much that for any prime $l\ge l_0$ we have $S(n,k,l)=1$, for a finite set of primes we compute $M(n,K,l)$ using the alternative descriptions of $m(K,l)$ and $t(K,l)$.

        \item \texttt{finiteFieldOrder:}\\
       	Find an ideal $\mathfrak p$ in the ring of integers $\mathcal O_K$ with the ramification index $e(\mathfrak p)<p-q$ and determine the prime power $q=\mathcal N(\mathfrak p)$

        \item \texttt{subgroupsNonchecked:}\\
        Find all subgroups of $GL_n(\mathds{F}_{q})$, up to isomorphism, with order dividing Schur's bound, where $q$ is the prime power from the previous step.
		\bigskip

		\noindent
        Now, we have a list of finite groups that may be isomorphic to a subgroup of $\operatorname{GL}_{3}{K}$. Then for each subgroup $G$ in the list we perform the following:
        \bigskip

        \item \texttt{characterSumsOfFaithfulRepresentations:}\\
        Compute all irreducible characters of $G$ of degree not exceeding $n$ and add characters from an orbit of the Galois group of the field $K$ over $\QQ$ multiplied  by the Schur index (of any character from the orbit).
		Compute sums of obtained characters which are faithful and have degree equal $n$.

		\item \texttt{autGroups:}\\
		Compute automorphisms of the group $G$, for each automorphism determine the induced permutation on the set of conjugacy classes of $G$, remove from the list characters that differ by an action of a permutation induced by an automorphism.

		\bigskip

		\noindent Images of representations afforded by characters from the constructed list \texttt{resCT} are all finite subgroups of $\operatorname{GL}(n,K($)).

		\bigskip

		\item Decompose the characters from the list \texttt{irrResCT}. For every irreducible summand construct an irreducible $G$-modules over $K$
		afforded by it. This yields a list \texttt{irrGModules} of irreducible $G$-modules over $K$.
		Decompose every character in the list  \texttt{resCT} as sums of characters of $G$-modules from \texttt{irrGModules}.  Using this decomposition compute representations affording all characters in \texttt{resCT}.
    \end{enumerate}
    It might happen that the character of constructed representation differs from a given character in \texttt{resCT}, the reason is that the computation of irreducible $G$-modules can produce $G$-modules over a field different from $K$. In this situation we use a slower method where we compute Absolutely Irreducible Modules, and then apply Wei restriction.

By the Stark–Heegner theorem \cite{Stark} there are nine imaginary quadratic fields $K$ whose ring of integers $\mathcal O_K$ is a Principal Ideals Domain
in the table we list some results of computations:

\[\def\arraystretch{1.3}\def\arraystretch{1.3}
\begin{array}{c|c|c|c|c}
	n&K&S(n,K)&\parbox{4cm}{\centering\# finite subgroups of $\operatorname{GL}(n,K)$}
&\parbox{4cm}{\centering\# finite subgroups of  $\operatorname{SL}(n,K)$}\\\hline
2&\QQ&24&	10&5\\
3&\QQ&48&	32&11\\
4&\QQ&5760&	227&106\\
5&\QQ&11520&	955&226\\
3&\QQ[\sqrt{-1}]&384&178&28\\
3&\QQ[\sqrt{-2}]&96&48&16\\
3&\QQ[\sqrt{-3}]&1286&352&40\\
3&\QQ[\sqrt{-7}]&336&41&15\\
3&\QQ[\sqrt{-11}]&48&37&13\\
3&\QQ[\sqrt{-19}]&48&40&14\\
3&\QQ[\sqrt{-43}]&48&40&14\\
3&\QQ[\sqrt{-67}]&48&40&14\\
3&\QQ[\sqrt{-163}]&48&40&14\\
\end{array}\]

\subsection{Finite subgroups of $\operatorname{GL}(3,\QQ{\sqrt{-19}})$ and $\operatorname{SL}(3,\QQ{\sqrt{-19}})$}.

Although the Schur bound for $S(3,\QQ{\sqrt{-19}}) = S(3,\QQ) =48$, numbers of finite subgroups of  $\operatorname{GL}(3,\QQ{\sqrt{-19}})$ and $\operatorname{SL}(3,\QQ{\sqrt{-19}})$ are strictly larger that the numbers of finite subgroups of
$\operatorname{GL}(3,\QQ)$ and
$\operatorname{SL}(3,\QQ)$, respectively. Finite subgroups of $\operatorname{GL}(3,\QQ{\sqrt{-19}})$ are isomorphic to the following
\begin{multline*}
	C_1, C_2, C_2, C_2, C_3, C_4, C_4, C_2^2, C_2^2, C_2^2, S_3, S_3, C_6, C_6,
	C_6, C_2\times C_4, Q_8, Q_8, C_2^3,\\ D_4, D_4, D_4, D_6, D_6, D_6, D_6, C_3\rtimes C_4,
	C_3\rtimes C_4, A_4, C_2\times C_6, C_2\times Q_8, C_2\times D_4, C_2\times C_3:C_4, \\
	\operatorname {SL}(2,3), C_2^2\times S_3, S_4, S_4, C_2\times A_4, C_2\times \operatorname {SL}(2,3),
	C_2\times S_4
\end{multline*}
Two (non-isomorphic) groups in that list realize the Schur bound $S(3,\QQ{\sqrt{-19}})=48$
\begin{multline*}
C_2\times\operatorname{SL}(2,3) \cong
\left\langle \begin{pmatrix}
	1&0&0\\0&-2&\tfrac12(\sqrt{-19}-3)\\0&\tfrac12(\sqrt{-19}+3)&3
\end{pmatrix},
\begin{pmatrix}
	1&0&0\\0&\tfrac12(\sqrt{-19}+3) & 3\\0&\tfrac12(\sqrt{-19}+1)&\tfrac12(\sqrt{-19}-3)
\end{pmatrix},
\right.
\\
\left.
\begin{pmatrix}
	-1&0&0\\0&\tfrac12(-\sqrt{-19}+3)&\tfrac12(-\sqrt{-19}-1)\\
	0&-3&\tfrac12(\sqrt{-19}-3)
\end{pmatrix}.
\begin{pmatrix}
	1&0&0\\0&-1&0\\0&0&-1
\end{pmatrix}.
\begin{pmatrix}
	-1&0&0\\0&-1&0\\0&0&-1
\end{pmatrix}\right\rangle
\end{multline*}
and
\begin{multline*}
	C_2\times S_4 \cong
\left\langle 	\begin{pmatrix}
		-1& 0& 0\\
		0& 0& 1\\
		0& 1& 0
	\end{pmatrix},
	\begin{pmatrix}
		0& 1& 0\\
		0& 0& 1\\
		-1& 0& 0
	\end{pmatrix},
	\begin{pmatrix}
		1&-1& 1\\
		0& 0& 1\\
		0& 1& 0,
	\end{pmatrix},
	\begin{pmatrix}
		0&-1& 0\\
		-1& 0& 0\\
		-1& 1&-1
	\end{pmatrix},
	\begin{pmatrix}
		-1& 0& 0\\
		0&-1& 0\\
		0& 0&-1
	\end{pmatrix}
	\right\rangle
\end{multline*}
	The character of the group $C_2\times S_4$  equals
	\[\arraycolsep=1mm
	\begin{array}{l|cccccccccccccc}
		\chi_{1}& 1& 1& 1& 1& 1& 1& 1& 1& 1& 1& 1& 1& 1& 1 \\
		\chi_{2}& 1& -1& -1& 1& 1& 1& 1& -1& 1& -1& -1& -1& -1& 1 \\
		\chi_{3}& 1& 1& 1& 1& -\zeta_3 - 1& \zeta_3& 1& 1& -\zeta_3 - 1& \zeta_3&
		-\zeta_3 - 1& \zeta_3& -\zeta_3 - 1& \zeta_3 \\
		\chi_{4}& 1& 1& 1& 1& \zeta_3& -\zeta_3 - 1& 1& 1& \zeta_3& -\zeta_3 - 1&
		\zeta_3& -\zeta_3 - 1& \zeta_3& -\zeta_3 - 1 \\
		\chi_{5}& 1& -1& -1& 1& \zeta_3& -\zeta_3 - 1& 1& -1& \zeta_3& \zeta_3 + 1&
		-\zeta_3& \zeta_3 + 1& -\zeta_3& -\zeta_3 - 1 \\
		\chi_{6}& 1& -1& -1& 1& -\zeta_3 - 1& \zeta_3& 1& -1& -\zeta_3 - 1&
		-\zeta_3& \zeta_3 + 1& -\zeta_3& \zeta_3 + 1& \zeta_3 \\
		\chi_{7}& 2& 2& -2& -2& -1& -1& 0& 0& 1& 1& 1& -1& -1& 1 \\
		\chi_{8}& 2& -2& 2& -2& -1& -1& 0& 0& 1& -1& -1& 1& 1& 1 \\
		\chi_{9}& 2& -2& 2& -2& \zeta_3 + 1& -\zeta_3& 0& 0& -\zeta_3 - 1& -\zeta_3&
		\zeta_3 + 1& \zeta_3& -\zeta_3 - 1& \zeta_3 \\
		\chi_{10}& 2& 2& -2& -2& -\zeta_3& \zeta_3 + 1& 0& 0& \zeta_3& -\zeta_3 - 1&
		\zeta_3& \zeta_3 + 1& -\zeta_3& -\zeta_3 - 1 \\
		\chi_{11}& 2& 2& -2& -2& \zeta_3 + 1& -\zeta_3& 0& 0& -\zeta_3 - 1& \zeta_3&
		-\zeta_3 - 1& -\zeta_3& \zeta_3 + 1& \zeta_3 \\
		\chi_{12}& 2& -2& 2& -2& -\zeta_3& \zeta_3 + 1& 0& 0& \zeta_3& \zeta_3 + 1&
		-\zeta_3& -\zeta_3 - 1& \zeta_3& -\zeta_3 - 1 \\
		\chi_{13}& 3& -3& -3& 3& 0& 0& -1& 1& 0& 0& 0& 0& 0& 0 \\
		\chi_{14}& 3& 3& 3& 3& 0& 0& -1& -1& 0& 0& 0& 0& 0& 0
	\end{array}\]
	where $\zeta_3 = \exp(\frac{2\pi i}3)$ is a third root of unity.

	Characters $\chi_7$ and $\chi_8$ have the Schur index over $\QQ$ equal to 2, while their Schur index over $\QQ(\sqrt{-19})$ (also over $\QQ(\sqrt{-d})$ for all Heegner numbers except 7) equals 1. Kernels of $\chi_7$ and $\chi_8$ are groups of order 2 generated by matrices
	\[\left(\begin{array}{rrr}
		1&0&0\\0&-1&0\\0&0&-1
	\end{array}\right)
	\qquad \text{and} \qquad
	\left(\begin{array}{rrr}
		-1&0&0\\0&-1&0\\0&0&-1
	\end{array}\right)
	\]
	respectively. Sums of characters
	\[
	\begin{array}{c|cccccccccccccc}
		\chi_2+\chi_7 &  3& 1& -3& -1& 0& 0& 1& -1& 2& 0& 0& -2& -2& 2 \\
		\chi_2+\chi_8 &  3& -3& 1& -1& 0& 0& 1& -1& 2& -2& -2& 0& 0& 2
	\end{array}
	\]
	 are faithful. They differ by an automorphism of the group $C_2\times \operatorname{SL}(2,3)$, consequently they define different embedding of that group with equal images. Note, that the kernel of the character $\chi_2$ equals $\operatorname{SL}(2,3)$.

	 Finite subgroups of $\operatorname{GL}(3,\QQ(\sqrt{-19}))$  that are not conjugate to a subgroup of $\operatorname{GL(3,\QQ)}$ are isomorphic to
	 \[ Q_8, Q_8, C_3\rtimes C_4, C_3\rtimes C_4, C_2\times Q_8, C_2\times C_3\rtimes C_4, SL(2,3),
	 C_2\times \operatorname{SL}(2,3)
	 \]

\section{MAGMA}
\label{sec:magma}
We have implemented the algorithm presented in the previous section in the computer algebra system \texttt{MAGMA}
\lstset{language=Magma,breaklines=false,breakatwhitespace=true,basicstyle=\tiny, comment=[s]{/*}{*/},commentstyle=\tiny\emph, showstringspaces=false, keepspaces=true, breaklines=true,  tabsize=2 }

\begin{lstlisting}
schurBoundForPrime := function(n, l, field)
	local mseq, t, m, schurBoundRes, fieldl;
	if Degree(field) eq 1
        then m:=1; t:=l-1;
	else
        if l eq 2 then
            fieldl:=RelativeField(field,CompositeFields(field,CyclotomicField(4))[1]);
            t:=Degree(fieldl);
            m:=1;
            while IsSubfield(CyclotomicField(2^(m+1)),fieldl)
                do m:=m+1;
            end while;
            if m eq 2 and not IsSquare(field!-1) then
                m:=1;
                t:=1;
            end if;
            else
                fieldl:=RelativeField(field,CompositeFields(field,CyclotomicField(l))[1]);
                t:=Degree(fieldl);
                m:=1;
                while IsSubfield(CyclotomicField(l^(m+1)),fieldl)
                    do m:=m+1;
                end while;
            end if;
        end if;
        if l eq 2 then
            schurBoundRes:=2^(n-Floor(n/t));
        else
            schurBoundRes:=1;
        end if;
        schurBoundRes:=schurBoundRes*l^(m*Floor(n/t)+Valuation(Factorial(Floor(n/t)),l));
	return(schurBoundRes);
end function;

schurBound := function(n,field)
	local l;
	return(&*[schurBoundForPrime(n,l,field): l in PrimesUpTo(n*AbsoluteDegree(field)+2)]);
end function;

orderOfFiniteFieldForPrime := function(l,field)
	local primedIdeal,ramificationIndex;
	primeIdeal,ramificationIndex:=Explode(Factorization(l*Integers(field))[1]);
	if ramificationIndex lt l-1 then
        return (Norm(primeIdeal));
	else
        return -1;
	end if;
end function;

finiteFieldOrder := function(n,field)
	local L,n0,mB, orders,order;
	n0:=NextPrime(100*n*Degree(field)+5);
	orders:=[ orderOfFiniteFieldForPrime(ll,field) : ll in PrimesUpTo(n0)];
	orders:=[ ooo: ooo in orders|ooo ge 5];
	order:=Min(orders);
	return order;
end function;

subgroupsNonchecked:=function(n,field)
	local fford,mbound,subgroups, subgroupsI;
	fford:=finiteFieldOrder(n,field);
	mbound:=schurBound(n,field);
	subgroups:=[(G`subgroup): G in Subgroups(GL(n,fford):OrderDividing:=mbound)];
	for k:=#subgroups to 1 by -1 do
		for m:=k-1 to 1 by -1 do
				if IsIsomorphic(subgroups[k],subgroups[m]:SmallOuterAutGroup:=200000000) then
					Remove(~subgroups,k);
					break;
			end if;
		end for;
	end for;
    return subgroups;
end function;

function get_CF_subfield(CF, K) local L,S,t;
    if Degree(CF) eq 1 or Degree(K) eq 1 then
        return Rationals();
    end if;
    L := Factorization(PolynomialRing(CF) ! DefiningPolynomial(K));
    S := OptimizedRepresentation(sub<CF | &cat[Coefficients(t[1]): t in L]>);
    return S;				
end function;

characterSums:=function(n, CharacterList)
local CLSplit,result, fixedClasses, decomp, triples;
	CLSplit:=[];
	result:={};
	for k:=1 to n do
		Append(~CLSplit, [cl: cl in CharacterList|cl[1] eq k]);
	end for;
	fixedClasses:=[[{i: i in [1..#cl]| cl[i] eq cl [1]}: cl in CLSplit[k] ]: k in [1..n]];
	for part in Partitions(n) do
		ttuples:=CartesianProduct([[1..#fixedClasses[tttt]]:tttt in part]);
		tuples:=[[x: x in ttt]: ttt in ttuples];
		tuples:=[x: x in tuples|&meet[fixedClasses[part[i]][x[i]]: i in [1..#part]] eq {1}];
		result join:={&+[CLSplit[part[i]][tt[i]]: i in [1..#part]]: tt in tuples};
	end for;
	return SetToSequence(result);
end function;

charactersOfFaithfulRepresentations:=function(n,group,field)
	local  CT, schInd, irrCT, EirrCT, character, resCT, EresCT, newChar,  conjClasses, nClasses, autGroup,classMap,permGroup, charRing, resIrrCT, galOrb, schGalChar, characterSum, decCharacterSum, irrResCT;
	conjClasses:=ConjugacyClasses(group);
	nClasses:=Nclasses(group);
	autGroup:=AutomorphismGroup(group);
	classMap:=ClassMap(group);
	permGroup:=sub<Sym(nClasses)|[[classMap((autGroup.k)(conjClasses[i][3])): i in [1..nClasses]]: k in [1..Ngens(autGroup)]]>;
	charRing:=CharacterRing(group);
	CT:=CharacterTable(group);
	CT:=[ct: ct in CT| ct[1] le n];
	resCT:=[];
	schResCT:=[];
	resIrrCT:=[];
	irrCT:=CT;
	while #irrCT gt 0 do
		galOrb:=GaloisOrbit(irrCT[1],get_CF_subfield(CharacterField(irrCT[1]),field));
		schInd:=SchurIndex(irrCT[1],field);
		schGalChar:=schInd*(&+galOrb);
		if schGalChar[1] le n then
			Append(~resCT,schGalChar);
		end if;
		irrCT:=RemoveIrreducibles(galOrb, irrCT);
	end while;
	resCT:=characterSums(n,resCT);
	EresCT:=[[ccc[i]: i in [1..#ccc]]:ccc in resCT];
	k:=1;
	while  k lt #EresCT do
		cccc:={[EresCT[k][i^g]: i in  [1..#EresCT[k]]]: g in permGroup};
		for j:=#EresCT to k+1 by -1 do
			if  EresCT[j] in cccc
				then Remove(~EresCT,j);
			end if;
		end for;
		k:=k+1;
	end while;
	if #EresCT gt 0 then
		characterSum:=charRing![&+[EresCT[i][j]: i in [1..#EresCT]]: j in [1..nClasses]];
		decCharacterSum:=Decomposition(CT,characterSum);
		irrResCT:=[CT[ii]: ii in [1..#CT]|decCharacterSum[ii] ne 0];
	else
		irrResCT:=CT;
	end if;
	return EresCT,irrResCT,CT, charRing;
end function;

weilRestriction:=function (algNum,F,L)
	local tt;
	if L eq Rationals() then return Matrix(Rationals(),1,1,[algNum]); end if;
	tt:=IsSubfield(F,L);
	tt:=IsSubfield(Parent(algNum),L);
	return RepresentationMatrix(RelativeField(F,L)!algNum);
	end function;

weilRestrictionMatrix:=function(algmat,F,L)
	return BlockMatrix(Nrows(algmat),Ncols(algmat),[weilRestriction(algmat[i][j],F,L): i in [1..Nrows(algmat)], j in [1..Ncols(algmat)]]);
end function;

representationsFromCharacters :=function(n,group,field)
	local  EresCT, irrResCT, CT, charRing, representations, resCT, irrGModules, charIrrGModules, decomp, rest, newModule, absIrrModules, absCT, absDecomp, absRest, newAbsIrrMod, reprField, weilResModule, irrModules;
	EresCT,irrResCT,CT,charRing:=charactersOfFaithfulRepresentations(n,group,field);
	if #EresCT eq 0 then
        return [];
    end if;
	representations:=[];
	resCT:=[charRing!chr: chr in EresCT];
	irrGModules:=[(irrMod): irrMod in IrreducibleModules(group, field: Characters:=irrResCT)| Dimension(irrMod) le n];
	charIrrGModules:=[Character(irrMod):irrMod in irrGModules];
	for character in resCT do
        decomp,rest:=Decomposition(charIrrGModules, character);
		if rest eq Zero(charRing) then
			newModule:=DirectSum(MultisetToSequence({*irrGModules[i]^^Integers()!decomp[i]:i in [1..#decomp]*}));
			Include(~representations,sub<GL(n,field)|ActionGenerators(newModule)>);
		else
			printf "Warning: Using slow method\n";
			absIrrModules:=AbsolutelyIrreducibleModules(group);
			absCT:=[Character(module): module in absIrrModules];
			absDecomp,absRest:=Decomposition(absCT,character);
			newAbsIrrMod:=[];
			for i:=1 to #absIrrModules do
				reprField:=BaseRing(absIrrModules[i]);
				weilResModule:=GModule(group,[weilRestrictionMatrix(mat, Rationals(),reprField): mat in  ActionGenerators(absIrrModules[i])]);
				weilResModule:=ChangeRing(weilResModule,field);
				irrModules:=Decomposition(weilResModule);
                Include(~newAbsIrrMod,irrModules[1]);
            end for;
            newModule:=DirectSum(MultisetToSequence({*newAbsIrrMod[i]^^Integers()!absDecomp[i]:i in [1..#absDecomp]*}));
            if Character(newModule) eq character then
		Include(~representations,MatrixGroup(newModule));
	    else
                return "Reconstruction of representation from character failed, representation afforded by character %o is missing\n",character;
            end if;
        end if;
	end for;
	return representations;
end function;

gln:=function(n,field: progress:=false)
	local  resSubgroups, subGroups;
	if Degree(field) eq 1 then
		ffield:=QNF();
	else
		ffield:=field;
	end if;
	subGroups:=subgroupsNonchecked(n,field);
	nn:=1;
	resSubgroups:=[];
	for subgroupsNC in subGroups do
		if progress then
			printf "\b"^(Ceiling(Log(10,#subGroups))+31) cat "Checking subgroup No. %*o of %o  ",Ceiling(Log(10,#subGroups)),nn,#subGroups;
		end if;
        nn:=nn+1;
        resSubgroups cat:=representationsFromCharacters(n,subgroupsNC,field);
	end for;
	printf "\n\n";
	return (resSubgroups);
end function;

sln:=function(n,field: progress:=false)
	local  glnList,slnList;
	glnList:=gln(n,field: progress);
	slnList:=[g: g in glnList|forall{gg: gg in GeneratorsSequence(g)|Determinant(gg) eq 1}];
	return slnList,glnList;
end function;

\end{lstlisting}

\end{document}